\documentclass[reqno]{amsart}
\usepackage{amssymb,eucal,latexsym, mathrsfs,enumerate,xcolor}

\makeatletter
\@namedef{subjclassname@2020}{\textup{2020} Mathematics Subject Classification}
\makeatother

\newenvironment{enumeratei}{\begin{enumerate}[\upshape (i)]}%
{\end{enumerate}}
{\end{enumerate}}
\newenvironment{enumerater}{\begin{enumerate}[\upshape (1)]}%
{\end{enumerate}}
\newcommand{\conj}{\mathbin{\bigwedge\mkern-15mu\bigwedge}}

\newcommand{\pup}[1]{\textup{(}{#1}\textup{)}}

\newcommand{\eqdef}{\overset{\mathrm{def}}{=}}

\newcommand{\spd}[2]{\left({#1}\mid{#2}\right)}

\newcommand{\Mat}[2]{\operatorname{M}_{{#1}}({#2})}

\newcommand{\Jac}{\mathrm{J}}

\DeclareMathOperator{\Add}{Add}

\DeclareMathOperator{\card}{card}

\newcommand{\ga}{\alpha}

\newcommand{\gc}{\gamma}
\newcommand{\gd}{\delta}
\newcommand{\gf}{\varphi}

\newcommand{\gl}{\lambda}

\newcommand{\gS}{\Sigma}

\newcommand{\one}{\mathbf{1}}

\newcommand{\ol}[1]{\overline{#1}}

\newcommand{\pI}[1]{\bigl({#1}\bigr)}
\newcommand{\pII}[1]{\Bigl({#1}\Bigr)}

\newcommand{\set}[1]{\left\{#1\right\}}
\newcommand{\setm}[2]{\set{{#1}\mid{#2}}}
\newcommand{\vecm}[2]{({#1}\mid{#2})}

\newcommand{\seq}[1]{\langle{#1}\rangle}
\newcommand{\sseq}[1]{\langle\!\langle{#1}\rangle\!\rangle}

\newcommand{\ZZ}{\mathbb{Z}}
\newcommand{\QQ}{\mathbb{Q}}

\newcommand{\cL}{{\mathcal{L}}}

\newcommand{\cR}{{\mathcal{R}}}
\newcommand{\cS}{{\mathcal{S}}}
\newcommand{\cT}{{\mathcal{T}}}

\numberwithin{equation}{section}

\newtheorem*{stat}{\name}
\newcommand{\name}{testing}

\theoremstyle{plain}

\newtheorem{theorem}{Theorem}[section]
\newtheorem{proposition}[theorem]{Proposition}
\newtheorem{corollary}[theorem]{Corollary}
\newtheorem{lemma}[theorem]{Lemma}

\theoremstyle{definition}

\newtheorem{definition}[theorem]{Definition}
\newtheorem{notation}[theorem]{Notation}
\newtheorem{example}[theorem]{Example}
\newtheorem{problem}{Problem}

\theoremstyle{remark}
\newtheorem{remark}[theorem]{Remark}
\newtheorem*{note}{Note}

\newcommand{\qedc}{{\qed}~{\rm Claim~{\theclaim}.}}
\newcommand{\qedsc}{{\qed}~{\rm Claim.}}

\numberwithin{figure}{section}
\numberwithin{table}{section}

\newcommand{\vp}{\mathsf{p}}
\newcommand{\vq}{\mathsf{q}}

\newcommand{\vA}{\mathsf{A}}

\newcommand{\vE}{\mathsf{E}}

\newcommand{\vS}{\mathsf{S}}

\newcommand{\kk}{\Bbbk}

\title[Definability of addition]{Is addition definable from multiplication and successor?}

\author[F. Wehrung]{Friedrich Wehrung}
\address{Normandie Universit\'e, UNICAEN\\
CNRS UMR 6139, LMNO\\
14000 Caen\\
France}
\email{friedrich.wehrung01@unicaen.fr}
\urladdr{https://wehrungf.users.lmno.cnrs.fr}

\date{\today}

\subjclass[2020]{16R60; 16R20; 16R40; 16U50; 16E50; 16N20; 16D60; 16B70; 16Y30; 16Y60}

\keywords{Ring; associative; unital; commutative; functional identity; additive; multiplicative; successor; brachymorphism; brachynomial; summable; addable}

\begin{document}

\begin{abstract}
A map $f\colon R\to S$ between (associative, unital, but not necessarily commutative) rings is a \emph{brachymorphism} if $f(1+x)=1+f(x)$ and $f(xy)=f(x)f(y)$ whenever $x,y\in R$.
We tackle the problem whether every brachymorphism is additive (i.e., $f(x+y)=f(x)+f(y)$), showing that in many contexts, including the following, the answer is positive:
\begin{itemize}
\item
$R$ is finite (or, more generally, $R$ is left or right Artinian);

\item
$R$ is any ring of $2\times2$ matrices over a commutative ring;

\item
$R$ is Engelian;

\item
every element of~$R$ is a sum of $\pi$-regular and central elements (this applies to $\pi$-regular rings, Banach algebras, and power series rings);

\item
$R$ is the full matrix ring of order greater than~$1$ over any ring;

\item
$R$ is the monoid ring~$K[M]$ for a commutative ring~$K$ and a $\pi$-regular monoid~$M$;

\item
$R$ is the Weyl algebra~$\vA_1(K)$ over a commutative ring~$K$ with positive characteristic;

\item
$f$ is the power function $x\mapsto x^n$ over any ring;

\item
$f$ is the determinant function over any ring~$R$ of $n\times n$ matrices, with $n\geq3$, over a commutative ring, such that if $n>3$ then~$R$ contains~$n$ scalar matrices with non zero divisor differences.

\end{itemize}
%We leave open the problem whether every brachymorphism is additive.
\end{abstract}

\maketitle

\tableofcontents

\section{Introduction}\label{S:Intro}

The present paper investigates the natural question whether the addition in an associative, unital ring is determined by the multiplication together with the successor function $x\mapsto 1+x$.
While the question whether this can be done equationally has been extensively studied (cf. Subsection~\ref{Su:Bckgrd} for more details), the general question remains open.
Here we will focus on a notion of definability involving a weakening of the concept of homomorphism.
It will turn out (cf. Corollary~\ref{C:MainConj}) that this amounts to studying the positive primitive definability of addition from multiplication and successor.

\subsection{The problem}\label{Su:IntroPb}
A map $f\colon R\to S$ between (associative, unital, but not necessarily commutative) rings will be called a \emph{brachymorphism}%
\footnote{
The prefix \emph{brachy}, from the ancient Greek \emph{brakh\'ys}, means ``short''. (\emph{Source}: Wiktionary.)}
if $f(1+x)=1+f(x)$ and $f(xy)=f(x)f(y)$ whenever $x,y\in R$.
For every brachymorphism $f\colon R\to S$ and every $x\in R$, $f(x)f(0)=f(x0)=f(0)$, thus
 \[
 f(0)+f(0)=f(0)+f(1)f(0)=(1+f(1))f(0)=f(1+1)f(0)=f(0)\,,
 \]
thus $f(0)=0$, and thus also $f(1)=1$.
The question whether every brachymorphism is additive (and thus a ring homomorphism) remains a mystery to the author.
A positive solution to that problem would probably destroy the present paper.

We verify that the answer to the question is positive in many ``classical'' contexts, thus letting the concept of brachymorphism ring various bells including regular rings (in von~Neumann's sense), commutative rings (or, more generally, PI rings), and rings of matrices (over commutative rings).

\subsection{Background}\label{Su:Bckgrd}
Throughout the paper all our rings will be unital, but not necessarily commutative.
They will also be associative, except in our final Section~\ref{S:MoreCtxts} where more contexts, including right near-rings and cancellative semirings, will be discussed.

Rickart~\cite{Rick1948}, Johnson~\cite{Johnson1958}, Martindale~\cite{Mart1969b} state sufficient conditions implying that a multiplicative homomorphism, between rings, is additive.
Related functional equations, such as
 \[
 f(x+y)+f(xy)=f(x)+f(y)+f(x)f(y)\,,
 \]
are studied in Dhombres~\cite{Dhombres1988}.
Ger and Reich~\cite{GerReich2010} handle functional equations of the form
 \[
 af(xy)+bf(x)f(y)+cf(x+y)+df(x)+kf(y)=0\,,
 \]
and study when a function satisfying such an identity must be a ring homomorphism.
Functional identities are studied further in Bre\v{s}ar~\cite{Bres2023}, Bre\v{s}ar, Chebotar, and Martindale~\cite{BCM2007}; see also Moln\'{a}r~\cite{Moln2000} for related work oriented towards operator algebras.

However, none of those works seems to involve the preservation of the successor function (i.e., $f(1+x)=1+f(x)$).
Anticipating Definition~\ref{D:Unamb}, we say that a ring~$R$ is \emph{addable} if every brachymorphism with domain~$R$ is additive.

A well studied example of successor-preserving map is the \emph{Frobenius map}\linebreak $\gf\colon x\mapsto\nobreak x^p$ on any algebra~$R$ over a field of characteristic~$p>0$: indeed\linebreak $(1+x)^p=1+x^p$ whenever $x\in R$.
On the other hand, the ring-theoretical identity $(xy)^p=x^py^p$ follows from commutativity but does not imply it unless $p=2$ (cf. Johnsen \emph{et al.} \cite[Example~3]{JOY1968b}).
Riley proves in~\cite{Riley2017} that \emph{if~$R$ is Engelian \pup{definition recalled in Section~\textup{\ref{S:AddComm2Summ}}}, then some power of~$\gf$ is a ring homomorphism}.
Our Corollary~\ref{C:(x+1)^n} implies immediately that \emph{a power of~$\gf$ is a ring homomorphism if{f} it is a multiplicative homomorphism}.
We also verify (Corollary~\ref{C:Engel}) that \emph{every Engelian ring is addable}.

As far as we could check, the investigation of the definability of addition in terms of other operations including the successor $x\mapsto 1+x$ started in Foster~\cite{Fos1951b}.
In~\cite{Yaq1956}, Yaqub expresses the addition of any ring $\ZZ/n\ZZ$ as a composition of multiplication and the successor function (anticipating Definition~\ref{D:brachynomials}, we call such operations \emph{brachynomials}).
Moore and Yaqub~\cite{MooYaq1978} extend that result to any ring satisfying a polynomial identity of the form $x^n=x^{n+1}p(x)$ in which every idempotent is central.
Those results are further refined in Abu-Khuzam,Tominaga, and Yaqub~\cite{AKTY1980}.
The assumption that every idempotent is central is removed in Putcha and Yaqub~\cite{PutYaq1985}.
The result is extended to commutative rings in which every element is a sum of idempotents in Yaqub~\cite{Yaqub1981}.
Finally, Komatsu~\cite{Koma1982} proves that \emph{rings in which the addition is a composition of multiplication and the successor function are characterized by their satisfying a polynomial identity of the form $x^n=x^{n+1}p(x)$, and then the latter can be taken as $x^n=x^{2n}$ for a suitable~$n$}.

All the above-mentioned works establish that in certain classes of rings~$R$, the addition can be expressed as a brachynomial.
Although this implies (non-trivially!) that every brachymorphism from~$R$ is additive (Corollary~\ref{C:Komatsu}), the converse does not hold: for example, $\ZZ$ does not satisfy any identity of the form $x^n=x^{2n}$, nonetheless every brachymorphism from~$\ZZ$ is (trivially) additive.
The distinction in play boils down to the meaning of \emph{definable}: in all the above-mentioned works, ``definable'' functions are meant in the strongest possible sense (i.e., brachynomials), which leads to the description of addition \emph{via} \emph{quantifer-free formulas}.
By contrast, the statement that every brachymorphism from~$R$ is additive can be expressed in terms of certain \emph{positive existential formulas} that we call \emph{summability formulas} (Lemma~\ref{L:CharAdd}).
We are therefore reaching a much wider class of rings than the ones characterized by Komatsu as above --- in fact, so wide that we do not know so far whether there is anything outside.

\subsection{Addable rings and elements}\label{Su:AddRngElt}
As stated in the Abstract, the main aim of the paper is to establish a sample of addability results for large classes of rings, such as $\pi$-regular rings, Banach algebras, and commutative rings.
A ubiquitous byproduct of our discussion will be the concept of an \emph{addable element}.
An element~$x$ in a ring~$R$ is \emph{addable} if $f(x+y)=f(x)+f(y)$ whenever $y\in R$ and~$f$ is a brachymorphism with domain~$R$ (Definition~\ref{D:Unamb}).
We will verify that the set~$\Add{R}$ of all addable elements in~$R$ is an additive subgroup of~$R$ (Proposition~\ref{P:S0S1}), closed under multiplication by regular elements (Corollary~\ref{C:AddxReg}), containing the Jacobson radical of~$R$ (Corollary~\ref{C:JacRad}), the center of~$R$ (Corollary~\ref{C:pq2Add}), and all nilpotent elements of~$R$ (Corollary~\ref{C:x^n2xAdd}).
We will further verify that~$\Add{R}$ is right and left integrally closed in~$R$ (Theorem~\ref{T:IntClos}), and that if $xy-yx$ is addable for all $y\in R$ then~$x$ is addable (Corollary~\ref{C:z=x+y}).
As an illustration, we will verify the addability of certain monoid rings~$K[M]$, including the case where~$M$ is $\pi$-regular (Theorem~\ref{T:KMreg}), and of the Weyl algebra~$\vA_1(K)$ in case~$K$ has positive characteristic (Theorem~\ref{T:Weyl1}), where~$K$ denotes a commutative ring.

\section{The case of division rings: addition on a slide rule}
\label{S:DivRings}

Warming up to the context, we begin with a very easy observation.

\begin{proposition}\label{P:vrinDivRing}
Let~$R$ be a division ring and let~$S$ be a unital ring.
Then every brachymorphism $f\colon R\to S$ is additive.
\end{proposition}

\begin{proof}
Let $x,y\in R$ and set $z\eqdef x+y$; we must prove that $f(z)=f(x)+f(y)$.
If $x=0$ then this follows from the observation above that $f(0)=0$; thus suppose that $x\neq0$.
{F}rom $zx^{-1}=1+yx^{-1}$ it follows that $f(z)f(x)^{-1}=1+f(y)f(x)^{-1}$, whence $f(z)=f(x)+f(y)$.
\end{proof}

Proposition~\ref{P:vrinDivRing}, though trivial, is quite instructive, in that it showcases the method, familiar to math students from the pre-calculator era, for adding numbers on a slide rule: namely, apply the formula $x+y=(1+yx^{-1})x$.
A number of our ``brachymorphism $\Rightarrow$ additive'' results will be established by refinements of that argument, with various degrees of sophistication.

\section{Summability and addability: a model-theoretical take}
\label{S:SummAdd}

A central concept underlying our framework is the following.

\begin{definition}\label{D:Unamb}
A finite tuple $(a_1,\dots,a_n)$ of elements in a ring~$R$ is \emph{summable} if for every ring~$S$ and every brachymorphism $f\colon R\to S$, $f\pI{\sum_{i=1}^na_i}=\sum_{i=1}^nf(a_i)$.
An element~$a$ of~$R$ is \emph{addable} if the pair $(a,x)$ is summable for all $x\in R$.
We define the \emph{addable kernel} of~$R$ as the set~$\Add{R}$ of all addable elements of~$R$.
A subset~$X$ of~$R$ is addable if $X\subseteq\Add{R}$.
In particular, the ring~$R$ is addable if{f} $R=\Add{R}$.
\end{definition}

In particular, $0$, $1$, and more generally all integer multiples of the unit, are addable.
(Formally this follows from the forthcoming Proposition~\ref{P:S0S1}; however, checking it directly is a straightforward exercise.)

This section will be centered on Lemma~\ref{L:CharAdd}, which expresses summability in terms of satisfaction of certain positive existential formulas.
The main first-order vocabularies used throughout the paper are the following:
\begin{itemize}
\item
$\cR=(+,\cdot,0,1)$, with~$+$ (the ``addition'') and~$\cdot$ (the ``multiplication'') both binary operation symbols and~$0$, $1$ symbols of constant: the language of (unital) rings;

\item
$\cS=(',\cdot,0)$, with~$'$ a unary operation symbol, $\cdot$ a binary operation symbol, and~$0$ a symbol of constant.
\end{itemize}
In any ring, the symbols of~$\cR$ will be interpreted the usual way, while the operation~$'$ will be interpreted by the successor function $x\mapsto x'=1+x$.
In particular, for any rings~$R$ and~$S$, the $\cR$-homomorphisms from~$R$ to~$S$ are exactly the ring homomorphisms, while the $\cS$-homomorphisms from~$R$ to~$S$ are exactly the brachymorphisms.

Something we will find more convenient to work with than $\cR$-terms (i.e., compositions of~$+$, $\cdot$, $0$, $1$) are their associated (non-commutative) polynomials: for example, $x+y$ and~$y+x$ are not the same $\cR$-term, but they define the same polynomial.
Every $\cS$-term~$t$ can be interpreted as a polynomial~$\tilde{t}$, and every $\cS$-formula~$\vE$ can be interpreted as an $\cR$-formula~$\tilde{\vE}$:
formally, $\tilde{0}=0$, $\tilde{t'}=1+\tilde{t}$ (\emph{accordingly, let us from now on abbreviate the~$\cS$-term~$0'$ by~$1$}), and $\widetilde{t_1t_2}=\tilde{t}_1\tilde{t}_2$, whenever~$t$, $t_1$, $t_2$ are $\cS$-terms.
If~$\vE$ is $t_1=t_2$, then~$\tilde{\vE}$ is $\tilde{t}_1=\tilde{t}_2$.
In particular, for any ring~$R$ and any $\cS$-sentence~$\vE$ with parameters from~$R$, $(R,+,\cdot,0,1)$ satisfies~$\tilde{\vE}$ if{f} $(R,',\cdot,0)$ satisfies~$\vE$.

\begin{definition}\label{D:brachynomials}
Polynomials of the form~$\tilde{t}$ will be called \emph{brachynomials}.
\end{definition}

For example, for distinct variables~$x$ and~$y$, $x+xy=x(1+y)$ is a brachynomial (it is~$\tilde{t}$ where~$t\eqdef xy'$) but~$x+y$ is not (exercise).

We will be using the vector notation $\vec{x}=(x_1,\dots,x_n)$ for strings of variables.
Let us recall a classical definition of model theory.

\begin{definition}\label{D:Atpppe}
A formula, in a first-order language without relation symbols (e.g., $\cR$ and~$\cS$), is
\begin{itemize}
\item
\emph{atomic} if it has the form $s=t$ for terms~$s$ and~$t$;

\item
\emph{positive primitive} if it has the form $(\exists\vec{x})\vE$ where~$\vE$ is a conjunction of atomic formulas;

\item
\emph{positive existential} if it has the form $(\exists\vec{x})\vE$ where~$\vE$ is a disjunction of conjunctions of atomic formulas.
\end{itemize}
\end{definition}

The proof of the following lemma is a standard argument of elementary model theory;
we include it for completeness.
The reader unfamiliar with the language of model theory may skip the proof of the ``only if'' direction in Lemma~\ref{L:CharAdd}, as that direction will not be needed through the remainder of the paper.

\begin{lemma}\label{L:CharAdd}
A finite tuple $(a_1,\dots,a_n)$ of elements in a ring~$R$ is summable if{f} there exists a positive existential $\cS$-formula~$\vE(x_1,\dots,x_n,x_{n+1})$ such that:
\begin{enumerater}
\item\label{SuffSum}
Every ring satisfies the implication%
\footnote{
By convention, the universal quantifiers are skipped: the formula is meant to begin with~$\forall\vec{x}$.
}
$\tilde{\vE}(\vec{x})\Rightarrow x_{n+1}=x_1+\cdots+x_n$;

\item\label{NecessSum}
$R$ satisfies $\tilde{\vE}(a_1,\dots,a_n,a_1+\cdots+a_n)$.
\end{enumerater}
Furthermore, $\vE$ can be taken positive primitive.
\end{lemma}

A formula~$\vE$ satisfying~\eqref{SuffSum}, \eqref{NecessSum} above will thus be called a \emph{summability formula} for $(a_1,\dots,a_n)$ (in~$R$).

\begin{proof}
Set $a_{n+1}\eqdef a_1+\cdots+a_n$ and suppose first that there exists a formula~$\vE$ as stated.
Let $f\colon R\to S$ be a brachymorphism between rings.
Since~$\vE$ is positive existential and~$f$ is a brachymorphism, it follows from~\eqref{NecessSum} that~$S$ satisfies $\tilde{\vE}(f(a_1),\dots,f(a_n),f(a_{n+1}))$, which, by~\eqref{SuffSum}, entails $f(a_{n+1})=f(a_1)+\cdots+f(a_n)$.

Suppose, conversely, that $(a_1,\dots,a_n)$ is summable within~$R$.
Denote by~$\cL_R$ the vocabulary obtained by adjoining to~$\cR$ a collection of extra symbols of constant~$\dot{b}$ for $b\in R$, and by~$\setm{\vA_i}{i\in I}$ the collection of all atomic $\cS$-formulas such that~$R$ satisfies $\tilde{\vA}_i\vecm{b}{b\in R}$.

Denote by~$\cT$ the set of~$\cL_R$-sentences obtaining by adjoining, to the axioms of ring theory, all atomic $\cR$-formulas~$\tilde{\vA}_i\vecm{\dot{b}}{b\in R}$ for $i\in I$.
Thus the models of~$\cT$ are in one-to-one correspondence with the brachymorphisms with domain~$R$ (interpret each~$\dot{b}$ by~$f(b)$).
Our assumption that $(a_1,\dots,a_n)$ is summable with sum~$a_{n+1}$ thus means exactly that~$\cT\vdash\dot{a}_{n+1}=\dot{a}_1+\cdots+\dot{a}_n$ (where~$\vdash$ denotes entailment).
By the compactness Theorem, there are a finite subset~$J$ of~$I$ and a finite subset~$B$ of~$R$ containing $\set{a_1,\dots,a_n,a_{n+1}}$, such that, denoting  conjunction over a set of formulas by~$\conj$,
 \begin{equation}\label{Eq:brachyimpl}
 (\text{ring theory})+
 \conj_{j\in J}\tilde{\vA}_j\vecm{\dot{b}}{b\in B}\ \vdash\ 
 \dot{a}_{n+1}=\dot{a}_1+\cdots+\dot{a}_n\,.
 \end{equation}
Setting $\vec{y}\eqdef\vecm{y_b}{b\in B}$, the formula~$\vE(x_1,\dots,x_n,x_{n+1})$ defined as
 \[
 (\exists\vec{y})\pII{\conj_{1\leq i\leq n+1}
 (x_i=y_{a_i})\ \&\ 
 \conj_{j\in J}\vA_j\vecm{y_b}{b\in B}}
 \]
is positive primitive, and by~\eqref{Eq:brachyimpl},
 \[
 (\text{ring theory})+\tilde{\vE}(\vec{x})\ \vdash\ x_{n+1}
 =x_1+\cdots+x_n\,.
 \]
Moreover, by the definition of the~$\vA_i$, $R$ satisfies $\tilde{\vE}(a_1,\dots,a_n,a_{n+1})$ (interpret each~$y_b$ by~$b$).
\end{proof}

\begin{example}\label{Ex:SumFlaDivRing}
The summability formula~$\vE(x,y,z)$ underlying the proof of Proposition~\ref{P:vrinDivRing} is $(x=0\ \&\ z=y)\text{ or }(\exists u)(ux=1\ \&\ zu=(yu)')$.
Indeed, modulo ring theory, $\tilde{\vE}(x,y,z)$ is equivalent to
 \[
 (x=0\ \&\ z=y)\text{ or }(\exists u)(ux=1\ \&\ zu=1+yu)\,,
 \]
which entails $z=x+y$.
Conversely, in any division ring, if $z=x+y$, then $\tilde{\vE}(x,y,z)$ is satisfied (if $x\neq0$ set $u\eqdef x^{-1}$).

The formula~$\vE$ above is positive existential, but not positive primitive: this is because it is designed to take all possible $(x,y)$ into account, which leads us to separate the cases $x=0$ and $x\neq0$.
%There are various ways to overcome that annoyance, which we will encounter in the forthcoming sections.
\end{example}

The observations above yield the following reformulation of our main conjecture.

\begin{corollary}\label{C:MainConj}
The following statements are equivalent:
\begin{enumeratei}
\item\label{MainConj}
Every brachymorphism between rings is additive.

\item\label{MainConjZZ}
The pair $(x,y)$ is summable in $\ZZ\seq{x,y}$ \pup{the free ring on two generators}.

\item\label{MainConjLogic}
There exists a positive existential \pup{resp., positive primitive} $\cS$-formula~$\vE(x,y,z)$ such that every ring satisfies the equivalence $\tilde{\vE}(x,y,z)\Leftrightarrow z=x+y$.

\end{enumeratei}
\end{corollary}

\begin{proof}
\eqref{MainConj}$\Rightarrow$\eqref{MainConjZZ} is trivial.

\eqref{MainConjZZ}$\Rightarrow$\eqref{MainConjLogic}.
Suppose that~\eqref{MainConjZZ} holds.
Let a positive primitive $\cS$-formula $\vE(x,y,z)$ witness the summability of $(x,y)$ in~$\ZZ\seq{x,y}$ \emph{via} Lemma~\ref{L:CharAdd}.
Since $\ZZ\seq{x,y}$ satisfies the positive existential sentence $\tilde{\vE}(x,y,x+y)$ of ring theory, every ring satisfies $(\forall x,y)\tilde{\vE}(x,y,x+y)$.
Since every ring satisfies the implication $\tilde{\vE}(x,y,z)\Rightarrow z=x+y$, it satisfies the required equivalence.

\eqref{MainConjLogic}$\Rightarrow$\eqref{MainConj}.
If~\eqref{MainConjLogic} holds, then, by the easy direction of Lemma~\ref{L:CharAdd}, every pair of elements, in any ring, is addable.
\end{proof}

The formulation~\eqref{MainConjLogic}, in Corollary~\ref{C:MainConj}, implies that if the conjecture had a positive solution, then such a solution could be, in principle, accessible through number crunching.
However, in view of the size of the search space being at least a double exponential (in the degrees of the polynomials involved), that option looks quite unpractical \emph{a priori}.

\section{Basic preservation results for addable elements}\label{S:Reg2Add}

An element~$x$ in a monoid~$M$ is \emph{regular}\label{page:regular}
if $xtx=x$ for some~$t\in M$, then called a \emph{quasi-inverse} of~$x$.
Of course, every left or right unit of~$M$ is regular.
We shall use that concept mainly for multiplicative monoids of rings, except in Theorem~\ref{T:KMreg} where~$M$ will be a general monoid.

This section will be centered on Proposition~\ref{P:S0S1} (addable elements form an additive subgroup) and Corollary~\ref{C:AddxReg} (addable times regular implies addable).
The addability of many ``classical'' rings, including but not restricted to von~Neumann regular rings and Banach algebras, will be stated in Corollary~\ref{C:Reg2Add}.

\begin{lemma}\label{L:f(nx)}
Let~$R$ and~$S$ be rings and let $f\colon R\to S$ be a brachymorphism.
Then $f(nx)=nf(x)$ whenever $x\in R$ and $n\in\ZZ$.
\end{lemma}

\begin{proof}
We already verified in the Introduction that $f(0)=0$, so it suffices to check the case where either $n>0$ or $n<0$.
For the first case, write $n=1+\cdots+1$ ($n$ times) and apply distributivity.
For the second case, it thus suffices to verify the case where $n=-1$.
{F}rom $0=1+(-1)$ and $f(0)=0$ it follows that $0=1+f(-1)$, thus $f(-1)=-1$, and thus $f(-x)=f((-1)x)=f(-1)f(x)=-f(x)$.
\end{proof}

\begin{proposition}\label{P:S0S1}
For every ring~$R$, $\Add{R}$ is an additive subgroup of~$R$.
\end{proposition}

\begin{proof}
It is obvious that~$0$ is addable and that the sum of any two addable elements is addable.
Hence it suffices to verify that~$-x$ is addable whenever~$x$ is addable.
Let~$f$ be a brachymorphism with domain~$R$.
Since~$x$ is addable, the equality $f(x)+f(-x+y)=f(y)$ holds whenever $y\in R$; whence, using Lemma~\ref{L:f(nx)},\newline $f(-x+y)=-f(x)+f(y)=f(-x)+f(y)$.
\end{proof}

\begin{lemma}\label{L:S1}
Let~$x$ and~$y$ be elements in a ring~$R$.
If the element~$x+y$ is addable, then the pair $(x,y)$ is summable.
\end{lemma}

\begin{proof}
We compute, using the addability of~$x+y$ together with Lemma~\ref{L:f(nx)},
\begin{equation*}
 f(x) + f(y) = f((x+y)+(-y)) + f(y) = f(x+y) + f(-y) + f(y) = f(x+y)\,.
 \tag*{\qed}
 \end{equation*}
 \renewcommand{\qed}{}
 \end{proof}

Although it is trivial that the product of an addable element by a unit is addable, the more general statement, whether the product of two addable elements is addable, is unclear.
In that direction, we could only achieve the following observation.

\begin{proposition}\label{P:xyv2xuv}
Let~$x$, $y$, $u$, $v$ be elements in a ring~$R$ such that $u=uvu$.
If $(x,yv)$ is summable, then $(xu,y)$ is summable.
\end{proposition}

\begin{proof}
Let~$f$ be a brachymorphism with domain~$R$ and set $z\eqdef xu+y$.
Using the summability of $(x,yv)$, we get
 \begin{equation}\label{Eq:f(zvu)}
 f(z)f(vu)=f(xu+yvu)=f(x+yv)f(u)=(f(x)+f(yv))f(u)=f(xu)+f(y)f(vu)\,.
 \end{equation}
On the other hand, from $f(-vu)=-f(vu)$ (cf. Lemma~\ref{L:f(nx)}) it follows that\linebreak $f(1-vu)=1-f(vu)$, thus, since $z(1-vu)=y(1-vu)$, we get
 \begin{equation}\label{Eq:f(z1-vu)}
 f(z)(1-f(vu))=f(y)(1-f(vu))\,.
 \end{equation}
Adding together the equations~\eqref{Eq:f(zvu)} and~\eqref{Eq:f(z1-vu)}, we get $f(z)=f(xu)+f(y)$.
\end{proof}

In particular, if~$u$ and~$v$ are mutual quasi-inverses, then $(xu,y)$ is summable if{f} $(x,yv)$ is summable.
Even more in particular, if $e^2=e$, then $(xe,ye)$ is summable if{f} $(xe,y)$ is summable if{f} $(x,ye)$ is summable.

\begin{corollary}\label{C:AddxReg}
Let~$a$ and~$u$ be elements in a ring~$R$.
If~$a$ is addable and~$u$ is regular, then~$au$ and~$ua$ are both addable.
\end{corollary}

\begin{proof}
By symmetry it suffices to prove that~$au$ is addable.
Let~$v\in R$ such that $u=uvu$.
For every $x\in R$, $(a,xv)$ is summable (because~$a$ is addable), thus, by Proposition~\ref{P:xyv2xuv}, $(au,x)$ is summable, as desired.
\end{proof}

Bringing together Proposition~\ref{P:S0S1} and Corollary~\ref{C:AddxReg}, we obtain that \emph{every sum of products of regular elements is addable}.
The following corollary sums up a few particular cases of that observation.

\begin{corollary}\label{C:Reg2Add}\hfill
\begin{enumerater}
\item\label{vnRR2Add}
Every von~Neumann regular ring is addable.

\item\label{PS2Add}
Every formal power series ring $\kk\sseq{\gS}$, with~$\kk$ a division ring and~$\gS$ a set, is addable.

\item\label{C*2Add}
Every \pup{unital} Banach algebra, viewed as a ring, is addable.

\item\label{Matn2Add}
For any ring~$R$ and any integer $n>1$, the full matrix ring~$\Mat{n}{R}$ is addable.

\end{enumerater}
\end{corollary}

\begin{proof}
Item~\eqref{vnRR2Add} follows immediately from the observation above.

In all rings considered in \eqref{PS2Add}--\eqref{Matn2Add}, every element is a sum of units, as follows:
\begin{itemize}
\item
For~\eqref{PS2Add}, recall that~$\kk\sseq{\gS}$ is the $\kk$-algebra of formal power series over~$\gS$ with coefficients from~$\kk$ (it is non-commutative unless $\card\gS\leq1$).
The units of~$\kk\sseq{\gS}$ are exactly the power series with nonzero constant term.
Hence, every~$a\in\kk\sseq{\gS}$ which is not a unit is the sum of two units (viz., $a=(-1)+(1+a)$).

\item
For~\eqref{C*2Add} this is because for every element~$x$ and every  real number $\ga>\|x\|$, $x=\ga+(x-\ga)$ is the sum of two units.

\item
By Raphael~\cite{Raph1974}, for any ring~$R$ and any $n>1$, every element of~$\Mat{n}{R}$ is a sum of~$2n^2$ units; this result got strengthened in Henriksen~\cite{Henr1974}, with three units.
At any rate, we are thus done for~\eqref{Matn2Add}.\qed
\end{itemize}
\renewcommand{\qed}{}
\end{proof}

It follows from Proposition~\ref{P:S0S1} that every ring, in which every element is a sum of units, is addable.
For a survey of such rings, see Srivastava~\cite{Sriv2010}.

\begin{corollary}\label{C:JacRad}
For every ring~$R$, the Jacobson radical~$\Jac(R)$ of~$R$ is addable in~$R$.
\end{corollary}

\begin{proof}
Every $x\in\Jac(R)$ can be written as $x=(-1)+(1+x)$, a sum of two units.
\end{proof}

\begin{corollary}\label{C:SimpleIdp}
Let~$R$ be a simple ring with a nontrivial idempotent.
Then~$R$ is addable.
\end{corollary}

\begin{proof}
By Herstein \cite[Corollary, p.~18]{Hers1969}, every element in the ring is a linear combination, with coefficients in~$\ZZ$, of products of idempotent elements.
By Corollary~\ref{C:AddxReg}, every such product is addable.
Apply Proposition~\ref{P:S0S1}.
\end{proof}

For an extension of Corollary~\ref{C:SimpleIdp}, see the comments following Corollary~\ref{C:CommIdp}.

\section{Homomorphic images, directed colimits, and finite products}
\label{S:ClAddSumm}

This section will be essentially devoted to verifying the preservation of summability by brachymorphic images, directed colimits, and finite products of rings.
The arguments are straightforward, with the (mild) exception of the case of finite products, which requires the additional lemma below.

\begin{lemma}\label{L:Sperp}
The $\cS$-formula $\vS_{\perp}(x,y,z)$ defined as $xy=0\ \&\ z'=x'y'$, is a summability formula for any pair $(x,y)$ of elements satisfying $xy=0$, in any ring.
\end{lemma}

\begin{proof}
A straightforward calculation: $\tilde{\vS}_{\perp}(x,y,z)$ is $xy=0\ \&\ 1+z=(1+x)(1+y)$, which is equivalent (in rings) to $xy=0\ \&\ z=x+y$.
\end{proof}

For a generalization of Lemma~\ref{L:Sperp} to a much wider class of pairs, see Lemma~\ref{L:pq2Add}.

\goodbreak
\begin{proposition}\label{P:ClosAdd}\hfill
\begin{enumerater}
\item\label{Quot2Summ}
Let~$R$ and~$S$ be rings and let $f\colon R\to S$ be a brachymorphism.
For every summable tuple $(a_1,\dots,a_n)$ in~$R$, the family $(f(a_1),\dots,f(a_n))$ is summable in~$S$.

\item\label{FinProd2Summ}
Let~$A$ and~$B$ be rings, let~$n$ be a positive integer, and let $(a_1,\dots,a_n)$ and $(b_1,\dots,b_n)$ be summable tuples in~$A$ and~$B$, respectively.
Then the tuple $((a_1,b_1),\dots,(a_n,b_n))$ is summable in $A\times B$.
\end{enumerater}
\end{proposition}

\begin{proof}
The proof of~\eqref{Quot2Add} is straightforward and we omit it.
As to~\eqref{FinProd2Add}, let~$R$ be a ring and let $f\colon A\times B\to R$ be a brachymorphism.
Set $a_{n+1}\eqdef a_1+\cdots+a_n$ and $b_{n+1}\eqdef b_1+\cdots+b_n$.
Then the elements $u\eqdef f(1,0)$ and $v\eqdef f(0,1)$ are idempotent elements in~$R$, with $uv=0$; they are the unit elements of the corner rings~$uRu$ and~$vRv$, respectively.
Let $f_A\colon A\to uRu$, $x\mapsto f(x,0)$ and $f_B\colon B\to vRv$, $y\mapsto f(0,y)$.
For each $x\in A$, $f_A(x)=f(x,0)=f(x,0)u$, thus
 \[
 f_A(1+x)=f(1+x,0)=f(1+x,1)f(1,0)=(1+f(x,0))u=u+f(x,0)u=u+f_A(x)\,,
 \]
so~$f_A$ is a brachymorphism.
Likewise, $f_B$ is a brachymorphism.
Since $(a_1,\dots,a_n)$ is summable in~$A$ and $(b_1,\dots,b_n)$ is summable in~$B$, it follows that
 \begin{equation}\label{Eq:fAfBAdd}
 f_A(a_{n+1})=\sum_{i=1}^nf_A(a_i)\text{ and }
 f_B(b_{n+1})=\sum_{i=1}^nf_B(b_i)\,.
 \end{equation}
Now for each $(x,y)\in A\times B$, $(x,0)(0,y)=(0,0)$, thus (cf. Lemma~\ref{L:Sperp}) $A\times B$ satisfies $\tilde{\vS}_{\perp}((x,0),(0,y),(x,y))$, and thus, since~$f$ is a brachymorphism, $R$ satisfies $\tilde{\vS}_{\perp}(f_A(x),f_B(y),f(x,y))$; whence $f(x,y)=f_A(x)+f_B(y)$.
Using~\eqref{Eq:fAfBAdd}, this entails $f(a_{n+1},b_{n+1})=\sum_{i=1}^nf(a_i,b_i)$.
\end{proof}

\begin{corollary}\label{C:ClosAdd}\hfill
\begin{enumerater}
\item\label{Quot2Add}
Let~$R$ and~$S$ be rings and let $f\colon R\twoheadrightarrow S$ be a \emph{surjective} brachymorphism.
Then $f[\Add{R}]\subseteq\Add{S}$.
In particular, if~$R$ is an addable ring, then so is~$S$.

\item\label{FinProd2Add}
Let~$A$ and~$B$ be rings.
Then $(\Add{A})\times(\Add{B})=\Add(A\times B)$.
In particular, if~$A$ and~$B$ are both addable rings, then so is $A\times B$.

\item\label{Colim2Add}
Consider a directed colimit $\vecm{R,f_i}{i\in I}=\varinjlim\vecm{R_i,f_{i,j}}{i\leq j\text{ in }I}$ of rings and ring homomorphisms.
Then for every $i\in I$ and every $a\in\bigcap_{j\geq i}f_{i,j}^{-1}[\Add{R_j}]$, $f_i(a)$ is addable in~$R$.
In particular, if all rings~$R_i$ are addable, then so is~$R$.
\end{enumerater}
\end{corollary}

\begin{proof}
Both~\eqref{Quot2Add} and the containment $(\Add{A})\times(\Add{B})\subseteq\Add(A\times B)$ follow immediately from Proposition~\ref{P:ClosAdd}.
Let $(a,b)\in A\times B$, let~$R$ be a ring, and let $f\colon A\to R$ be a brachymorphism.
The map $g\colon A\times B\to R$, $(x,y)\mapsto f(x)$ is also a brachymorphism, thus if $(a,b)$ is addable in~$A\times B$, then for all $x\in R$, $f(a+x)=g((a,b)+(x,0))=g(a,b)+g(x,0)=f(a)+f(x)$; whence $a\in\Add{A}$.
Similarly, $b\in\Add{B}$, thus completing the proof of~\eqref{FinProd2Add}.

It remains to verify~\eqref{Colim2Add}.
We must prove that for every $\ol{b}\in R$, the pair $(f_i(a),\ol{b})$ is summable in~$R$.
Since the poset~$I$ is upward directed, there are $j\geq i$ and $b\in R_j$ such that $\ol{b}=f_j(b)$.
By assumption on~$a$, the pair $(f_{i,j}(a),b)$ is summable in~$R_j$.
By Proposition~\ref{P:ClosAdd}\eqref{Quot2Summ}, the pair $(f_i(a),\ol{b})=(f_jf_{i,j}(a),f_j(b))$ is summable in~$R$.
\end{proof}

\begin{remark}\label{Rk:Ore}
The following observation, due to George Bergman, shows that \emph{if there exists a non-addable ring, then there exists a non-addable Ore ring}.
The construction runs as follows.
We know that~$\ZZ\seq{x,y}$ can be embedded into a division ring~$D$.
Form the polynomial ring~$D[t]$ in a central indeterminate~$t$ and let~$R$ be the subring of~$D[t]$ consisting of all polynomials~$f$ such that $f(0)\in\ZZ\seq{x,y}$.
Then~$R$ is both right and left Ore.
Indeed, $D[t]$ is both right and left Ore (cf. the comments following Cohn \cite[Proposition~0.10.3]{Cohn1985}), and for any $p,q\in R\setminus\set{0}$, if we let $pa = qb \neq 0$ in~$D[t]$, then $ta,tb\in R$ and $pta=qtb\neq0$.
But~$R$ has $\ZZ\seq{x,y}$ as a homomorphic image (via $f\mapsto f(0)$), hence, by Corollary~\ref{C:ClosAdd}, if $\ZZ\seq{x,y}$ is non-addable (which, by Corollary~\ref{C:MainConj}, is an equivalent form of the existence of a non-addable ring), then so is the Ore ring~$R$.
\end{remark}

\section{{F}rom commutativity to summability}\label{S:Comm2Summ}

This section's main result, Lemma~\ref{L:pq2Add}, provides a summability formula on any pair of ring elements satisfying a certain weak commutation relation.
In particular, \emph{every commutative ring is addable} (Corollary~\ref{C:pq2Add}).

\begin{notation}\label{Not:vpvq}
We introduce $\cS$-terms~$\vp_1$ and~$\vq_1$ as follows:
 \begin{align*}
 \vp_1(x,y,z)&\eqdef(xz)'(yz)'\,;\\
 \vq_1(x,y,z)&\eqdef((xy)'z^2)'\,. 
 \end{align*}
 \end{notation}
 
The brachynomials associated to~$\vp_1$ and~$\vq_1$ are thus the following:
 \begin{align*}
 \tilde{\vp}_1(x,y,z)&=(1+xz)(1+yz)\,;\\
 \tilde{\vq}_1(x,y,z)&=1+(1+xy)z^2\,. 
 \end{align*}
We shall now generalize an observation, stated at the beginning of Chapter~24 in Boolos, Burgess, and Jeffrey~\cite{BoBuJe2007}, that the addition over nonnegative integers can be defined in the language~$\cS$ by
 \[
 z=x+y\ \Longleftrightarrow\ \vp_1(x',y,z')=\vq_1(x',y,z')\quad
 \text{whenever }x,y,z\in\ZZ^+\,.
 \]
In any ring~$R$, define the \emph{commutator operation} by
 \begin{equation}\label{Eq:Commutator}
 [x,y]\eqdef xy-yx\,,\quad\text{whenever }x,y\in R\,.
 \end{equation}
As a straightforward calculation shows, the brachynomials~$\tilde{\vp}_1$ and~$\tilde{\vq}_1$ are related through the following equation, valid in any ring:
 \begin{equation}\label{Eq:tildep-q1}
 \tilde{\vp}_1(x,y,z)=\tilde{\vq}_1(x,y,z)+(x+y-z)z+x[z,y]z\,.
 \end{equation}
 
\begin{lemma}\label{L:pq2Add}
The $\cS$-formula~$\vS_{\mathrm{comm}}(x,y,z)$ defined as
 \[
 xzy=xyz\ \&\ \vp_1(x,y,z)=\vq_1(x,y,z)\ \&\ \vp_1(x,y',z')=\vq_1(x,y',z')
 \]
is a summability formula at any pair $(x,y)$, in any ring, such that $xyx=x^2y$.
 \end{lemma}

\begin{proof}
We first verify that every triple $(x,y,z)$, of elements in a ring~$R$, satisfying $\tilde{\vS}_{\mathrm{comm}}(x,y,z)$, satisfies $z=x+y$.
It follows from our assumptions that $x[z,y]=x[1+z,1+y]=0$, thus, by~\eqref{Eq:tildep-q1},
 \[
 (x+y-z)z=0\text{ and }\pI{x+(1+y)-(1+z)}(1+z)=0\,.
 \]
Since the latter equation simplifies to $(x+y-z)(1+z)=0$, it follows that\linebreak $x+y-z=0$, as required.

Let, conversely, $x,y\in R$ such that $xyx=x^2y$, and set $z\eqdef x+y$.
Since $[z,y]=[x,y]=xy-yx$, we get $x[z,y]=x(xy-yx)=0$, thus also $x[1+z,1+y]=x[z,y]=0$.
It follows that $xzy=xyz$ and $x(1+z)(1+y)=x(1+y)(1+z)$.
By~\eqref{Eq:tildep-q1}, it follows that $\tilde{\vp}_1(x,y,z)=\tilde{\vq}_1(x,y,z)$ and $\tilde{\vp}_1(x,1+y,1+z)=\tilde{\vq}_1(x,1+y,1+z)$.
\end{proof}

\begin{corollary}\label{C:pq2Add}
The following statements hold, for every ring~$R$:
\begin{enumerater}
\item\label{Comm2Summ}
Any tuple of pairwise commuting elements in~$R$ is summable.

\item\label{Cen2Add}
The center of~$R$ is an addable subset of~$R$.

\item\label{Comm2Add}
If~$R$ is commutative, then it is addable.

\end{enumerater}
\end{corollary}

A mild tweaking of the methods of this section will lead, in Section~\ref{S:AddComm2Summ}, to further strengthenings of Corollary~\ref{C:pq2Add}.

\section{Integral closedness of the addable kernel}\label{S:ClAddR}

The main result of this section, Theorem~\ref{T:IntClos}, states an integral closedness property of the addable kernel of any ring.
Consequences will be stated, such as: if $n>0$ and~$x^n$ is addable, then so is~$x$ (cf. Corollary~\ref{C:x^n2xAdd}), and: if the addition of~$R$ is given by a brachynomial then~$R$ is addable (Corollary~\ref{C:Komatsu}).
This will also enable us to state addability of all $K$-subalgebras of~$\Mat{n}{K}$ for commutative~$K$ (Corollary~\ref{C:MatnKAlg}).

\begin{definition}\label{D:AlgCl}
Let~$A$ be a subset in a ring~$R$.
An element~$x$ of~$R$ is \emph{right integral} over~$A$, with degree a positive integer~$n$, if there are elements $a_0,\dots,a_{n-1}\in A$ such that $x^n=\sum_{k<n}a_kx^k$.
We say that~$A$ is \emph{right integrally closed in~$R$} if every element of~$R$ which is right integral over~$A$ belongs to~$A$.
\end{definition}

Of course, ``left integral'' can be defined similarly, \emph{via} equations of the form $x^n=\sum_{k<n}x^ka_k$, and then, since addability is left-right symmetric, ``right'' can be replaced by ``left'' everywhere in this section.

A straightforward application of the binomial Theorem yields the following.

\begin{lemma}\label{L:IntTransInv}
Let~$A$ be an additive subgroup of a ring~$R$ such that $1\in A$.
If an element~$x$ of~$R$ is right integral over~$A$, then so is $1+x$, with the same degree.
\end{lemma}

\begin{lemma}\label{L:f(poly)}
Let~$n$ be a nonnegative integer, let~$x\in R$, and let $a_0,\dots,a_n\in\nobreak R$ with~$a_k$ addable whenever $k<n$.
Then for every ring~$S$ and every brachymorphism $f\colon R\to\nobreak S$,  $f(a_0+a_1x+\cdots+a_nx^n)=f(a_0)+f(a_1)f(x)+\cdots+f(a_n)f(x)^n$.
\end{lemma}

\begin{proof}
A straightforward induction argument, based on the formula
 \begin{equation*}
 a_0+a_1x+\cdots+a_nx^n=a_0+(a_1+a_2x+\cdots+a_nx^{n-1})x\,.
 \tag*{\qed}
 \end{equation*}
\renewcommand{\qed}{}
\end{proof}

\begin{corollary}\label{C:AddRoverx}
Let~$y$ be an element in a ring~$R$.
If every element of~$R$ has the form $\sum_{k=0}^na_ky^k$ where each $a_k\in\Add{R}$, then~$R$ is addable.
\end{corollary}

\begin{proof}
Any elements~$a$ and~$b$ of~$R$ can be written $a=\sum_{k=0}^na_ky^k$ and $b=\sum_{k=0}^nb_ky^k$, with all $a_k,b_k\in\Add{R}$, respectively.
Therefore, for any brachymorphism~$f$ with domain~$R$,
 \begin{align}
 f(a)+f(b)&=\sum_{k=0}^nf(a_k)f(y)^k+\sum_{k=0}^nf(b_k)f(y)^k&&
 (\text{apply Lemma~\ref{L:f(poly)}})\notag\\
 &=\sum_{k=0}^n\pI{f(a_k)+f(b_k)}f(y)^k\notag\\
 &=\sum_{k=0}^nf(a_k+b_k)f(y)^k
 &&(\text{because each }a_k\text{ is addable})\notag\\
 &=f(a+b)&&
 (\text{apply again Lemma~\ref{L:f(poly)}})\,.\tag*{\qed}
 \end{align}
\renewcommand{\qed}{}
\end{proof}

\begin{theorem}\label{T:IntClos}
For every ring~$R$, $\Add{R}$ is right integrally closed in~$R$.
\end{theorem}

\begin{proof}
Let~$x\in R$ be right integral over~$\Add{R}$, so $x^{n+1}=a_0+a_1x+\cdots+a_nx^n$ for some nonnegative integer~$n$ and $a_0,\dots,a_n\in\Add{R}$.
Let~$S$ be a ring, let $f\colon R\to S$ be a brachymorphism, and let $y\in R$.
Setting $z\eqdef x+y$ and~$\ol{z}\eqdef f(x)+f(y)$, we need to verify that $f(z)=\ol{z}$.
First, an immediate application of Lemma~\ref{L:f(poly)} yields
 \begin{equation}\label{Eq:f(x)integral}
 f(x)^{n+1}=f(a_0)+f(a_1)f(x)+\cdots+f(a_n)f(x)^n\,.
 \end{equation}
Moreover, from the equations
 \begin{equation}\label{Eq:zx^ny}
 zx^n=x^{n+1}+yx^n=\pII{\sum\nolimits_{k<n}a_kx^k}+(a_n+y)x^n
 \end{equation}
it follows that
 \begin{align*}
 f(z)f(x)^n&=f(zx^n)\\
 &=\pII{\sum\nolimits_{k<n}f(a_k)f(x)^k}+f(a_n+y)f(x)^n
 &&(\text{apply Lemma~\ref{L:f(poly)}})\\
 &=\pII{\sum\nolimits_{k<n}f(a_k)f(x)^k}+(f(a_n)+f(y))f(x)^n
 &&(\text{because }a_n\text{ is addable})\\
 &=\pII{\sum\nolimits_{k\leq n}f(a_k)f(x)^k}+f(y)f(x)^n\\
 &=f(x)^{n+1}+f(y)f(x)^n&&
 (\text{apply~\eqref{Eq:f(x)integral}})\\
 &=\ol{z}f(x)^n\,,
 \end{align*}
whence
 \begin{equation}
 (f(x)+f(y)-f(x+y))f(x)^n=(\ol{z}-f(z))f(x)^n=0\,.\label{Eq:olz-f(z)x^n}
 \end{equation}
Since $\Add{R}$ is an additive subgroup of~$R$ (cf. Proposition~\ref{P:S0S1}) containing~$1$, $1+x$ is also integral with degree~$n+1$ over~$\Add{R}$ (cf. Lemma~\ref{L:IntTransInv}).
Hence~\eqref{Eq:olz-f(z)x^n} also applies to the pair $(1+x,y)$, which, since that change turns~$f(x)$ to~$1+f(x)$ but does not affect the value of $f(x)+f(y)-f(x+y)$, leads to the equation
 \begin{equation}\label{Eq:olz-f(z)x+1^n}
 (f(x)+f(y)-f(x+y))(1+f(x))^n=0\,.
 \end{equation}
Pick $u,v\in\ZZ[t]$ such that
 \[
 t^nu(t)+(1+t)^nv(t)=1\,.
 \]
Then multiplying~\eqref{Eq:olz-f(z)x^n} on the right by~$u(f(x))$, \eqref{Eq:olz-f(z)x+1^n} on the right by~$v(f(x))$, and adding the two resulting equations together, we get $f(x)+f(y)-f(z)=0$, as desired.
\end{proof}

Let us now reap a few consequences of Theorem~\ref{T:IntClos}.

\begin{corollary}\label{C:MatnKAlg}
Let~$K$ be a commutative ring, let~$n$ be a positive integer, and let~$R$ be a subring of~$\Mat{n}{K}$.
If~$R$ contains all coefficients of the characteristic polynomial of each of its elements, then~$R$ is addable.
\end{corollary}

\begin{proof}
It follows from our assumption, together with the Cayley-Hamilton Theorem, that every element of~$R$ is right integral over the subring~$K'$ of all its scalar matrices.
Since~$K'$ is central in~$R$, every element of~$K'$ is addable in~$R$ (cf. Corollary~\ref{C:pq2Add}\eqref{Cen2Add}).
Apply Theorem~\ref{T:IntClos}.
\end{proof}

In particular, Corollary~\ref{C:MatnKAlg} applies to all subrings of~$\Mat{n}{K}$ containing all scalar matrices, also to \emph{trace algebras} (in characteristic zero) as considered in Procesi~\cite{Procesi1987}.
Further addability results, for rings of matrices not necessarily containing all scalar matrices, will be stated in Corollaries~\ref{C:TriangMatdet} and~\ref{C:M2(comm)}, together with Theorems~\ref{T:deta+1} and~\ref{T:deton3x3}.

\begin{corollary}\label{C:x^n2xAdd}
Let~$x$ be an element in a ring~$R$ and let~$n$ be a positive integer.
If~$x^n$ is addable, then so is~$x$.
In particular, every nilpotent element of~$R$ is addable.
\end{corollary}

Invoking Proposition~\ref{P:S0S1} and Corollary~\ref{C:AddxReg}, together with Corollary~\ref{C:pq2Add}, we obtain:

\begin{corollary}\label{C:abs2z=x+y}
For every element~$x$ in a ring~$R$, if some power of~$x$ is a sum of products of regular elements and central elements, then~$x$ is addable.
\end{corollary}

\begin{corollary}\label{C:Komatsu}
If the addition of a ring~$R$ can be expressed as a brachynomial, then~$R$ is addable.
\end{corollary}

\begin{proof}
By Komatsu~\cite{Koma1982}, there exists a positive integer~$n$ such that $x^n=x^{2n}$ whenever~$x\in R$.
In particular, $x^n$ is idempotent, thus regular.
By Corollary~\ref{C:abs2z=x+y}, $x$ is addable.
\end{proof}

\begin{note}
Off-hand, Corollary~\ref{C:Komatsu} looks (misleadingly) trivial: if $f\colon R\to S$ is a brachymorphism and~$t$ is a brachynomial representing the addition in~$R$ (i.e., $x+y=t(x,y)$ on~$R$), then $f(x+y)=f(t(x,y))=t(f(x),f(y))$ whenever $x,y\in R$; and then we could be tempted to conclude the proof by stating $t(f(x),f(y))=f(x)+f(y)$.
However, there is no reason \emph{a priori} for the latter argument to work, because we do not know whether the addition of~$S$ is represented by the brachynomial~$t$.
\end{note}

A monoid~$M$ is \emph{$\pi$-regular} if for every $x\in M$ there exists a positive integer~$n$ such that~$x^n$ is regular in~$M$ (let us then say that~$x$ is $\pi$-regular in~$M$).
A ring is $\pi$-regular if its multiplicative monoid is $\pi$-regular (cf. McCoy~\cite{McCoy1939}, Kaplansky~\cite{Kapl1950}, Azumaya~\cite{Azu1954}).
Trivially, every von~Neumann regular ring is $\pi$-regular.
So is every finite ring, in fact in that case every element has an idempotent power.
Kaplansky observes in~\cite{Kapl1950} that every algebraic algebra over a field is $\pi$-regular.

\begin{corollary}\label{C:pireg}
Every $\pi$-regular element in a ring is addable.
In particular, every $\pi$-regular ring is addable.
\end{corollary}

\begin{remark}\label{Rk:pireg}
It follows from Corollary~\ref{C:pireg} that \emph{every finite ring is addable}.
More can be said:
Dischinger proves in~\cite{Disch1976} that in any ring~$R$, if every descending chain $aR\supseteq a^2R\supseteq\cdots$ stabilizes, then every descending chain $Ra\supseteq Ra^2\supseteq\cdots$ also stabilizes, so~$R$ is \emph{strongly $\pi$-regular} in the sense of Azumaya~\cite{Azu1954}, thus $\pi$-regular.
In particular, \emph{every \pup{left or right} Artinian ring is $\pi$-regular, thus addable}.
\end{remark}

\section{{F}rom addable commutators to summable pairs}
\label{S:AddComm2Summ}

The main result of this section, Theorem~\ref{T:z=x+y}, implies that for any elements~$x$, $y$ in any ring, if the commutator $[x,y]$ is addable, then the pair~$(x,y)$ is summable.
Consequences of that result are then recorded for rings of triangular matrices (over a commutative ring; see Corollary~\ref{C:TriangMatdet}), rings of $2\times2$ matrices,
and certain PI rings (Corollaries~\ref{C:NoM2Id} and~\ref{C:Engel}) including rings satisfying an identity $(xy)^n=x^ny^n$ for some $n\geq2$ (Corollary~\ref{C:(x+1)^n}).

Theorem~\ref{T:z=x+y} will be obtained by extending some of the methods of Section~\ref{S:Comm2Summ}.
Although there is some overlap between the methods and results of Sections~\ref{S:Comm2Summ} and~\ref{S:AddComm2Summ}, none of those sections seems to imply the other \emph{a priori}.

\begin{lemma}\label{L:f(addcomm)}
The following statements hold, for any elements~$x$, $y$ in a ring~$R$ and $z\eqdef x+y$:
\begin{enumerater}
\item\label{[xy]yx}
The pair $([x,y],yx)$ is summable if{f} $f([x,y])=[f(x),f(y)]$ for every brachymorphism~$f$ of domain~$R$.

\item\label{[xz]zx}
The pair $([x,y],x^2+yx)$ is summable if{f} $f([x,z])=[f(x),f(z)]$ for every brachymorphism~$f$ of domain~$R$ if{f} the pair $([x,1+y],x^2+(1+y)x)$ is summable.
\end{enumerater}
\end{lemma}

\begin{proof}
\emph{Ad}~\eqref{[xy]yx}.
A straightforward calculation, based on the equation $xy=[x,y]+yx$.

\emph{Ad}~\eqref{[xz]zx}.
By~\eqref{[xy]yx}, the pair $([x,y],x^2+yx)=([x,z],zx)$ is summable if{f} $f([x,z])=[f(x),f(z)]$ for every brachymorphism~$f$ of domain~$R$.
Apply that observation to the triple $(x,1+y,1+z)$, using the relations $[x,z]=[x,1+z]$ and $[f(x),f(z)]=[f(x),1+f(z)]$.
\end{proof}

Tweaking the $\cS$-terms~$\vp_1$, $\vq_1$ introduced in Notation~\ref{Not:vpvq}, we set

\begin{notation}\label{Not:vpvq2}
We introduce $\cS$-terms~$\vp_2$ and~$\vq_2$ as follows:
 \begin{align*}
 \vp_2(x,y,z)&\eqdef(zx)'(yz)'\,;\\
 \vq_2(x,y,z)&\eqdef(z(xy)'z)'\,. 
 \end{align*}
 \end{notation}
 
The brachynomials associated to~$\vp_2$ and~$\vq_2$ are thus the following:
 \begin{align*}
 \tilde{\vp}_2(x,y,z)&=(1+zx)(1+yz)\,;\\
 \tilde{\vq}_2(x,y,z)&=1+z(1+xy)z\,. 
 \end{align*}
As a straightforward calculation shows, the brachynomials~$\tilde{\vp}_2$ and~$\tilde{\vq}_2$ are related through the following equation, valid in any ring:
 \begin{equation}\label{Eq:tildep-q2}
 \tilde{\vp}_2(x,y,z)+[x,z]=\tilde{\vq}_2(x,y,z)+(x+y-z)z\,.
 \end{equation}

\begin{theorem}\label{T:z=x+y}
Let~$x$ and~$y$ be elements in a ring~$R$.
If all pairs $([x,y],x^2+yx)$, $([x,y],\tilde{\vp}_2(x,y,x+y))$, and $([x,y],\tilde{\vp}_2(x,1+y,1+x+y))$ are summable, then the pair $(x,y)$ is summable.
In particular, if the element $[x,y]$ is addable, then the pair $(x,y)$ is summable.
\end{theorem}

\begin{proof}
Let~$f$ be a brachymorphism with domain~$R$ and set $z\eqdef x+y$.
The pair $([x,z],\tilde{\vp}_2(x,y,z))=([x,y],\tilde{\vp}_2(x,y,z))$ is, by assumption, summable.
Further, by Lemma~\ref{L:f(addcomm)}, $f([x,z])=[f(x),f(z)]$.
Hence,
 \begin{align*}
 \tilde{\vp}_2(f(x),f(y),f(z))+[f(x),f(z)]&=f\pI{\tilde{\vp}_2(x,y,z)}+f([x,z])\\
 &=f\pI{\tilde{\vp}_2(x,y,z)+[x,z]}\\
 &=f\pI{\tilde{\vq}_2(x,y,z)}&&(\text{apply~\eqref{Eq:tildep-q2}})\\
 &=\tilde{\vq}_2(f(x),f(y),f(z))\,.
 \end{align*}
On the other hand, by evaluating~\eqref{Eq:tildep-q2} at $(f(x),f(y),f(z))$, we get
 \[
 \tilde{\vp}_2(f(x),f(y),f(z))+[f(x),f(z)]=\tilde{\vq}_2(f(x),f(y),f(z))+(f(x)+f(y)-f(z))f(z)\,,
 \]
thus, subtracting from the above, we get
 \begin{equation}\label{Eq:fx+y-z2}
 (f(x)+f(y)-f(z))f(z)=0\,.
 \end{equation}
By Lemma~\ref{L:f(addcomm)}, together with our assumption, the argument above can be applied to the triple $(x,1+y,1+z)$, thus yielding
 \begin{equation}\label{Eq:fx+y-z2'}
 (f(x)+f(y)-f(z))(1+f(z))=0\,.
 \end{equation}
Subtracting~\eqref{Eq:fx+y-z2} from~\eqref{Eq:fx+y-z2'} yields $f(z)=f(x)+f(y)$.
\end{proof}

Further tweaking $\cS$-terms, we get the following curious observation.

\begin{theorem}\label{T:z+xzy}
A map $f\colon R\to\nobreak S$, between rings, is a ring homomorphism if{f}
 \begin{align}
 f(1+x)&=1+f(x)\,,\label{Eq:f(x+1)2}\\
 f(z+xzy)&=f(z)+f(x)f(z)f(y)\label{Eq:f(z+xzy)2}
 \end{align}
whenever $x,y\in R$ and $z\in\set{0,1,x+y}$.
\end{theorem}

\begin{proof}
We verify the nontrivial direction.
We shall invoke the following identity%
\footnote{
It is the particular case, for $n=2$, of a collection of identities, given in Cohn \cite[Lemma~3.1(ii)]{Cohn1963b}, relating \emph{continuant polynomials}.
}, similar to~\eqref{Eq:tildep-q1}:
 \begin{equation}\label{Eq:c3}
 1+(x+y+xzy)z=(1+xz)(1+yz)\,,
 \end{equation}
valid in every ring.
Now set $z\eqdef x+y$.
The equation~\eqref{Eq:c3} then specializes to
 \begin{equation}\label{Eq:c3bis}
 1+(z+xzy)z=(1+xz)(1+yz)\,.
 \end{equation}
Set $o\eqdef f(0)$ and $e\eqdef f(1)$.
Substituting $(0,0,0)$ then $(1,0,1)$ for $(x,y,z)$ in~\eqref{Eq:f(z+xzy)2}, we get, after cancellation,
 \[
 o^3=e^2o=0\,,
 \]
which, since $e=o+1$ (evaluate~\eqref{Eq:f(x+1)2} at~$0$) yields $o=0$, thus $e=1$.
Setting $z=1$ in~\eqref{Eq:f(z+xzy)2}, it follows that for all $x,y\in R$, $f(1+xy)=1+f(x)f(y)$, which, by~\eqref{Eq:f(x+1)2}, yields $f(xy)=f(x)f(y)$.
Therefore~$f$ is a brachymorphism.

Now let $x,y\in R$ and set $z\eqdef x+y$.
Applying the brachymorphism~$f$ to the equation~\eqref{Eq:c3bis} and using our assumption, we get
 \begin{equation}\label{Eq:c3ter}
 1+(f(z)+f(x)f(z)f(y))f(z)=(1+f(x)f(z))(1+f(y)f(z))\,.
 \end{equation}
On the other hand, evaluating the identity~\eqref{Eq:c3} at $(f(x),f(y),f(z))$, we get
 \[
 1+(f(x)+f(y)+f(x)f(z)f(y))f(z)=(1+f(x)f(z))(1+f(y)f(z))\,.
 \]
Owing to~\eqref{Eq:c3ter} and after cancellation, we get
 \[
 (f(x)+f(y)-f(z))f(z)=0\,,
 \]
thus also, applying that result to the triple $(1+x,y,1+z)$,
 \[
 (f(x)+f(y)-f(z))(1+f(z))=0\,,
 \]
and therefore, after cancellation, $f(x)+f(y)-f(z)=0$.
\end{proof}

Let us now reap a few consequences of Theorem~\ref{T:z=x+y}.

\begin{corollary}\label{C:z=x+y}
Let~$x$ be an element in a ring~$R$.
If all commutators $[x,y]$, for $y\in R$, are addable in~$R$, then~$x$ is addable in~$R$.
\end{corollary}

By Proposition~\ref{P:S0S1} and Corollary~\ref{C:x^n2xAdd}, we thus obtain the following.

\begin{corollary}\label{C:CommIdp}
Every ring in which every commutator is a sum of nilpotent elements is addable.
\end{corollary}

In particular, it follows from Gardella and Thiel~\cite[Theorem~5.3]{GarThi2023} that the assumption of Corollary~\ref{C:CommIdp} is satisfied by all \emph{zero-product balanced rings} as defined in~\cite{GarThi2023}.
By \cite[Proposition~3.6]{GarThi2023}, the class of such rings includes all rings generated by their idempotents.
In \cite[Example~3.7]{GarThi2023} (crediting Bre\v{s}ar~\cite{Bres2007}), it is observed that that class includes, in turn, the class of all simple rings with a nontrivial idempotent, or more generally the class of all rings~$R$ containing an idempotent~$e$ such that $ReR$ and $R(1-e)R$ both generate~$R$ as an additive subgroup.
For the former class, Corollary~\ref{C:CommIdp} is superseded by Corollary~\ref{C:SimpleIdp}.

\begin{corollary}\label{C:TriangMatdet}
Let~$n$ be a positive integer and let~$K$ be a commutative ring.
Then every ring~$R$ of $n\times n$ lower triangular matrices over~$K$ is addable.
\end{corollary}

\begin{proof}
For all $a,b\in R$, the commutator~$[a,b]$ is a lower triangular matrix with~$0$ on the diagonal; thus $[a,b]^n=0$.
By Corollary~\ref{C:x^n2xAdd}, $[a,b]$ is addable.
Since this holds for all $a,b\in R$, the desired conclusion follows from Corollary~\ref{C:z=x+y}.
\end{proof}

\begin{corollary}\label{C:M2(comm)}
Every ring of $2\times2$ matrices over a commutative ring is addable.
\end{corollary}

\begin{proof}
Every such ring satisfies Hall's identity $(xy-yx)^2z=z(xy-yx)^2$ (saying that every commutator has central square).
By Corollary~\ref{C:abs2z=x+y}, it follows that every commutator is addable.
Apply Corollary~\ref{C:z=x+y}.
\end{proof}

For $3\times3$ matrices over a commutative ring, we got only a partial answer so far, namely the case where the brachymorphism in question is the determinant function (Theorem~\ref{T:deton3x3}).
Assuming a sufficient supply of scalar units in our ring, the desired result holds for $n\times n$ matrices (Theorem~\ref{T:deta+1}).

For further results on rings of square matrices over commutative rings, see Section~\ref{S:det}.
Surprisingly, not every finite ring can be represented in this way:
Bergman~\cite{Berg74a} observes that for every prime number~$p$, the endomorphism ring of $(\ZZ/p^2\ZZ)\oplus(\ZZ/p\ZZ)$ is a finite ring that cannot be embedded into any full matrix ring over a commutative ring.

\begin{corollary}\label{C:R/Iadd}
Let~$I$ be a two-sided ideal in a ring~$R$.
If every element of~$I$ is addable in~$R$ and if $R/I$ is commutative, then~$R$ is addable.
\end{corollary}

\begin{proof}
Since~$R/I$ is commutative, all commutators from~$R$ belong to~$I$, thus, by assumption, they are addable.
By Corollary~\ref{C:z=x+y}, every element of~$R$ is addable.
\end{proof}

Other consequences of Theorem~\ref{T:z=x+y} are extensions of Corollary~\ref{C:pq2Add}\eqref{Comm2Add} to weaker identities than commutativity:

\begin{corollary}\label{C:NoM2Id}
Let a ring~$R$ satisfy a polynomial identity~$\vE$ satisfied neither by any non-commutative division ring nor by~$\Mat{2}{\kk}$ for any field~$\kk$.
Then~$R$ is addable.
\end{corollary}

\begin{proof}
The argument is the one used by Herstein in his proof of \cite[Theorem~1]{Hers1961}; we sketch it here for convenience.
The necessary background can be found in Lam \cite[Ch.~11]{Lam2001}.
If~$R$ is a division ring, then, since it satisfies~$\vE$, it is commutative.
If~$R$ is primitive then it has to be a division ring (thus, by the above, a field), otherwise some~$\Mat{2}{\kk}$, with~$\kk$ a field, would be a homomorphic image of a subring of~$R$, thus it would satisfy~$\vE$.
Hence, if~$R$ is semiprimitive then it is a subdirect product of fields, thus it is commutative.
Therefore, in the general case $R/{\Jac(R)}$ is commutative, which means that every commutator from~$R$ belongs to~$\Jac(R)$.
By Corollary~\ref{C:JacRad}, $\Jac(R)$ is addable in~$R$.
Apply Corollary~\ref{C:z=x+y}.
\end{proof}

Herstein proves in~\cite{Hers1961} that both polynomial identities $(xy)^n=x^ny^n$ and\linebreak $(x+y)^n=x^n+y^n$, where~$n$ is an integer greater than~$1$, satisfy the assumption of Corollary~\ref{C:NoM2Id}; hence any ring satisfying any of those is addable.
Focusing on the former identity, we obtain that the power function $x\mapsto x^n$, on any ring, cannot be a counterexample to our main conjecture.
More precisely:

\begin{corollary}\label{C:(x+1)^n}
For every positive integer~$n$, if a ring~$R$ satisfies the identity
 \begin{equation}
 (xy)^n=x^ny^n\,,\label{Eq:(xy)^n}
 \end{equation}
then it is addable.
In particular, if~$R$ also satisfies the identity
 \begin{equation}\label{Eq:(x+1)^n}
 (1+x)^n=1+x^n\,,
 \end{equation}
then it further satisfies the identity
 \begin{equation}\label{Eq:(x+y)^n}
(x+y)^n=x^n+y^n\,.
\end{equation}
\end{corollary}

Hence \emph{if the power function $x\mapsto x^n$ on~$R$ is a brachymorphism, then it is additive}.
This applies, in particular, to powers of the \emph{Frobenius map} $\gf\colon x\mapsto x^p$ on any ring of characteristic~$p$: hence, \emph{a power of~$\gf$ is a ring homomorphism if{f} it is a multiplicative homomorphism}.

Setting $[x,y]_1\eqdef[x,y]=xy-yx$ and $[x,y]_{n+1}\eqdef[[x,y]_n,y]$, a ring is \emph{Engelian} if it satisfies the identity $[x,y]_n=0$ for some~$n$, and \emph{locally Engelian} if every finitely generated subring is Engelian.

\begin{corollary}\label{C:Engel}
Every locally Engelian ring is addable.
\end{corollary}

\begin{proof}
By Corollary~\ref{C:ClosAdd}\eqref{Colim2Add}, it suffices to verify that every finitely generated Engelian ring~$R$ is addable.
Set $\gc^1(R)\eqdef R$, and define each $\gc^{n+1}(R)$ as the two-sided ideal of~$R$ generated by the commutators $[x,y]$ where $x\in\gc^n(R)$ and $y\in R$.
By Riley and Wilson \cite[Proposition~2]{RilWil1999}, $\gc^n(R)=0$ for some nonnegative integer~$n$.
Then a straightforward downward induction argument (over the largest~$k\leq n$ such that $x\in\gc^k(R)$), based on Corollary~\ref{C:z=x+y}, shows that every $x\in R$ is addable.
\end{proof}

\section{Monoid algebras and Weyl algebras}
\label{S:Weyl}

In this section we illustrate a few of our earlier results by verifying additivity in a few nontrivial cases, namely certain monoid rings (Theorem~\ref{T:KMreg}) and Weyl algebras in positive characteristic (Theorem~\ref{T:Weyl1}).

\begin{theorem}\label{T:KMreg}
Let~$K$ be a commutative ring and let~$M$ be a monoid.
We set
 \begin{align*}
 \ga_0(M)&\eqdef\setm{u\in M}{(\forall x\in M)(ux=xu)}
 \qquad(\text{the center of }M)\,,\\
 \ga_{n+1}(M)&\eqdef\{v\in M\mid
 v^m=x_1\cdots x_kuy_1\cdots y_l\\
 &\qquad\text{for some }m>0\,,\ 
 u\in\ga_n(M)\,,\ \text{all }x_i,y_j\text{regular}\}\quad
 \text{whenever }n\geq0\,.
 \end{align*}
If $M=\bigcup_n\ga_n(M)$, then the monoid ring~$K[M]$ is addable.
\end{theorem}

\begin{note}
In particular, Theorem~\ref{T:KMreg} specializes to the case where~$M$ is $\pi$-regular, or, more generally, where every element of~$M$ has a positive power which is a product of regular elements and a central element; that is, $M=\ga_1(M)$.
\end{note}

\begin{proof}
Set $N\eqdef\setm{v\in M}{(\forall\gl\in K)(\gl v\in\Add{K[M]})}$.
It suffices to verify that every $v\in M$ belongs to~$N$ (because then, by Proposition~\ref{P:S0S1}, every element of~$K[M]$, being a sum of addable monomials, would be addable in~$K[M]$).
We argue by induction on the least~$n$ such that $v\in\ga_n(M)$.
If $v\in\ga_0(M)$, then $\gl v$ is central in~$K[M]$ whenever $\gl\in K$, thus, by Corollary~\ref{C:pq2Add}, $\gl v$ is addable; whence $v\in N$.
Now suppose the property established at stage~$n$ and let $v\in\ga_{n+1}(M)$.
Pick $m>0$,\linebreak $u\in\ga_n(M)$, and regular $x_1,\dots,x_k,y_1,\dots,y_l\in M$ such that $v^m=x_1\cdots x_kuy_1\cdots y_l$.
It follows from our induction hypothesis that $\gl^mu\in\Add{K[M]}$ for every $\gl\in K$.
Invoking Corollary~\ref{C:AddxReg}, it follows that the element
$(\gl v)^m=x_1\cdots x_k\gl^muy_1\cdots y_l$ is addable in~$K[M]$.
By Corollary~\ref{C:x^n2xAdd}, it follows that~$\gl v$ is addable in~$K[M]$; whence $v\in\nobreak N$.
\end{proof}

\begin{theorem}\label{T:Weyl1}
Let~$K$ be a commutative ring.
If~$K$ has positive characteristic, then the Weyl algebra $\vA_1(K)\eqdef K\seq{x,y}/(xy-yx-1)$ is addable.
\end{theorem}

\begin{proof}
By assumption, there exists a positive integer~$m$ such that $m\cdot1=0$ within~$K$.
It follows that $x^my-yx^m=mx^{m-1}=0$ within~$\vA_1(K)$, so~$x^m$ commutes with~$y$.
Since it also commutes with~$x$, it follows that~$x^m$ is central.
Hence, for any $\gl\in K$ and any nonnegative integer~$k$, $(\gl x^k)^m=\gl^mx^{km}$ is also central, thus (cf. Corollary~\ref{C:pq2Add}) addable.
By Corollary~\ref{C:x^n2xAdd}, it follows that~$\gl x^k$ is addable.
By Proposition~\ref{P:S0S1}, it follows that every element of~$K[x]$ is addable in~$\vA_1(K)$.
Now every element of~$\vA_1(K)$ has the form $\sum_{k=0}^na_ky^k$ where each~$a_k\in K[x]$.
Since, by the above, all~$a_k$ are addable, it follows from Corollary~\ref{C:AddRoverx} that~$\vA_1(K)$ is addable.
\end{proof}

We do not know whether Theorem~\ref{T:Weyl1} extends to the case where~$K$ has characteristic zero.
In that direction, observe that regardless of the characteristic of~$K$, the pair $(x,y)$ is summable: indeed, the commutator $[x,y]=1$ is addable; apply Theorem~\ref{T:z=x+y}.
On the other hand, we do not know whether the pair $(x^2,y)$ is summable in~$\vA_1(\QQ)$.

\section{On the determinant function over rings of matrices}\label{S:det}

It is well known that the determinant function $\det\colon\Mat{n}{K}\to K$, for a positive integer~$n$ and a commutative ring~$K$, is multiplicative (i.e., $\det(ab)=\det(a)\det(b)$ for all $n\times n$ matrices~$a$ and~$b$).
On the other hand, the identity $\det(1+a)=1+\det(a)$ fails for trivial examples.
If by chance it holds over a subring~$R$ of~$\Mat{n}{K}$, then the question arises whether the determinant function is actually additive on~$R$.

In this section we tackle that problem.
We already know from Corollary~\ref{C:M2(comm)} that the answer to the question above is positive for $n\in\set{1,2}$, and that we then get full addability of~$R$.
In Theorem~\ref{T:deta+1} we will verify that the answer to the question above is also positive provided~$R$ contains at least~$n$ scalar matrices with non zero divisor differences, further establishing full addability of~$R$.
Finally, in Theorem~\ref{T:deton3x3} we will verify that this is also the case for $n=3$ (however we could not prove full addability in that case).

\begin{theorem}\label{T:deta+1}
Let~$n$ be a positive integer and let~$K$ a commutative ring.
Denote by~$\one$ the identity matrix in~$\Mat{n}{K}$ and let $\gl_1,\dots,\gl_n\in K$ such that $\gl_i-\gl_j$ is not a zero divisor in~$K$ whenever $i\neq j$.
If a subring~$R$ of~$\Mat{n}{K}$ contains $\setm{\gl_i\one}{1\leq i\leq n}$ and $\det(\one+a)=1+\det(a)$ whenever $a\in\nobreak R$, then~$a^n$ is central whenever $a\in R$.
In particular, $R$ is an addable ring, so the equation $\det(a+b)=\det(a)+\det(b)$ holds whenever $a,b\in R$.
\end{theorem}

For a similar result, with weaker assumptions, assuming $n=3$, see Theorem~\ref{T:deton3x3}.

\begin{proof}
Let us denote the characteristic polynomial of any $a\in\Mat{n}{K}$ as
 \begin{equation}\label{Eq:pa(t)}
 p_a(t)=\det(t\one-a)=
 t^n+\sum_{0<k<n}(-1)^{k}\tau_k(a)t^{n-k}+(-1)^n\det(a)\,,
 \end{equation}
where~$t$ is an indeterminate over~$K$ and each~$\tau_k$ is a polynomial in the entries of~$K$ over the integers.
(It is well known that each~$\tau_k(a)$ is the trace of the $k$th exterior power of~$a$.)
Our assumption means that the assignment $a\mapsto\det(a)$ defines a brachymorphism from~$R$ to~$K$.
By Corollary~\ref{C:pq2Add}, every central element of~$R$ is addable, thus
\[
p_a(\gl)=\det(\gl\one-a)=\gl^n+(-1)^n\det(a)\,,\quad
\text{whenever }\gl\in K\text{ and }\set{\gl\one,a}\subseteq R\,.
\]
Comparing with~\eqref{Eq:pa(t)} (evaluated at~$\gl_i$), it follows that
 \begin{equation}\label{Eq:truncdet}
 \sum_{0<k<n}(-1)^{k}\tau_k(a)\gl_i^{n-k}=0\qquad
 \text{whenever }1\leq i\leq n\text{ and }a\in R\,.
 \end{equation}
Now our hypothesis on the elements~$\gl_i-\gl_j$ implies that the Vandermonde determinant
 \[
 \begin{vmatrix}
 1 & \gl_1 & \cdots & \gl_1^{n-1}\\
 1 & \gl_2 & \cdots & \gl_2^{n-1}\\
 \vdots & \vdots & \ddots & \vdots\\
 1 & \gl_n & \cdots & \gl_n^{n-1}
 \end{vmatrix}
 =\prod_{1\leq i<j\leq n}(\gl_j-\gl_i)
 \]
is not a zero divisor in~$K$.
By~\eqref{Eq:truncdet}, it follows that $\tau_k(a)=0$ whenever $0<k<n$.
Due to~\eqref{Eq:pa(t)}, this means that
 \[
 p_a(t)=\det(t\one-a)=t^n+(-1)^n\det(a)\qquad
 \text{whenever }a\in R\,.
 \]
By the Cayley-Hamilton Theorem, it follows that
 \[
 a^n+(-1)^n\det(a)\one=0\qquad\text{whenever }a\in R\,.
 \]
In particular, $a^n$ is central.
The conclusion then follows from Corollary~\ref{C:abs2z=x+y}.
\end{proof}

Set $\spd{a}{b}\eqdef\tau_2(a+b)-\tau_2(a)-\tau_2(b)$, for all square matrices~$a$ and~$b$ of the same order over any commutative ring.
The following lemma is the special case where $n=2$ of the main result of Reutenauer and Sch\"{u}tzenberger~\cite{ReutSchutz1987}.
It can also easily be checked by a direct computation.

\begin{lemma}\label{L:tr(a)tr(b)2tr(ab)}
The equation $\tau_1(a)\tau_1(b)=\tau_1(ab)+\spd{a}{b}$ holds whenever~$a$ and~$b$ are square matrices of the same order over any commutative ring.
\end{lemma}

The next instance of the main result of Reutenauer and Sch\"{u}tzenberger~\cite{ReutSchutz1987}, obtained by letting $n=3$, states that
 \begin{multline}\label{Eq:det3x3}
 \tau_3(a+b)=
 \tau_3(a)+\tau_3(b)-\tau_1(ab)\tau_1(a+b)\\
 + \tau_1(a)\tau_2(b)+\tau_2(a)\tau_1(b)+
 \tau_1(a^2b)+\tau_1(ab^2)\,,
 \end{multline}
for any square matrices~$a$ and~$b$ of the same order over any commutative ring.
In what follows we will focus on $3\times3$ matrices, in which case~$\tau_3$ is the determinant function.

\begin{theorem}\label{T:deton3x3}
Let~$R$ be a ring of $3\times3$ matrices over a commutative ring.
If $\det(\one+a)=1+\det(a)$ whenever $a\in R$, then $\det(a+b)=\det(a)+\det(b)$ whenever $a,b\in R$.
\end{theorem}

\begin{proof}
For any $x\in R$, we get, evaluating the characteristic polynomial of~$x$ at~$1$,
 \[
 \det(\one-x)=p_{x}(1)=1-\tau_1(x)+\tau_2(x)-\det(x)\,,
 \]
so our assumption means that
 \begin{equation}\label{Eq:tr1=tr2}
 \tau_2(x)=\tau_1(x)\qquad\text{whenever }x\in R\,.
 \end{equation}
Since~$-x\in R$, \eqref{Eq:tr1=tr2} also holds at~$-x$, that is, 
$\tau_2(x)=-\tau_1(x)$.
It follows that
 \begin{equation}\label{Eq:2trx=0}
 2\tau_1(x)=0\qquad\text{whenever }x\in R\,.
 \end{equation}
By substituting $x+y$ to~$x$ in~\eqref{Eq:tr1=tr2}, then using the additivity of~$\tau_1$ and canceling out $\tau_1(x)+\tau_1(y)$, we obtain
 \begin{equation}\label{Eq:spd=0}
 \spd{x}{y}=0\qquad\text{whenever }x,y\in R\,.
 \end{equation}
By~\eqref{Eq:det3x3} combined with the identity $\tau_1(xy)=\tau_1(yx)$, it suffices to prove that for all $a,b\in R$, the scalar
 \[
 \gd\eqdef\tau_1(ab)\tau_1(a+b)-
 \tau_1(a)\tau_2(b)-\tau_2(a)\tau_1(b)-
 \tau_1(ab(a+b))
 \]
vanishes.
Since, by~\eqref{Eq:tr1=tr2} and~\eqref{Eq:2trx=0}, $\tau_1$ and~$\tau_2$ agree on~$R$, it follows from~\eqref{Eq:2trx=0} that
 \[
 \gd=\tau_1(u)\tau_1(v)-\tau_1(uv)
 \]
where we set $u\eqdef ab$ and $v\eqdef a+b$.
By Lemma~\ref{L:tr(a)tr(b)2tr(ab)}, this means that $\gd=\spd{u}{v}$.
Since~$u$ and~$v$ both belong to~$R$ and by~\eqref{Eq:spd=0}, it follows that $\gd=0$, as desired.
\end{proof}

\section{Expanding the context}
\label{S:MoreCtxts}

In order to keep our context consistent our discussion has been so far limited to (associative, unital) rings.
However, the concept of brachymorphism makes sense for more general structures, giving rise to non-trivial additional results and problems.

In what follows we will still be dealing with structures with signature $(0,1,+,\cdot)$.
Brachymorphisms will be defined by the rules $f(0)=0$ (which may no longer follow from the others), $f(1+x)=1+f(x)$, and $f(xy)=f(x)f(y)$.
A \emph{brachy-automorphism} will be defined as a bijective brachymorphism.
All counterexamples in this section were produced \emph{via} the \texttt{Mace4} software (cf. McCune~\cite{McCune}).

\subsection{Semirings}\label{Su:Semiring}
Recall that a structure $(R,0,1,+,\cdot)$ is a \emph{semiring} if $(R,+,0)$ is a commutative monoid, $(R,\cdot,1)$ is a monoid, and~$\cdot$ is distributive on~$+$\,.
For example, every ring, and also every bounded distributive lattice, is a semiring.
In the latter case, the addition is determined by the multiplication.
The semiring and brachy-automorphism represented in Table~\ref{Tab:BrachySemiring} show that in general, the addition of a commutative semiring may not be determined by its multiplication and successor function.
This example can be described as the $4$-element chain $0<a<b<1$, with $x+y=\max\set{x,y}$ for all elements~$x$, $y$, whereas $xy=0$ unless $1\in\set{x,y}$.

\begin{table}[htb]  \centering % size 4
\begin{tabular}{r|rrrr}
$+$ & 0 & $a$ & $b$ & 1\\
\hline
    0 & 0 & $a$ & $b$ & 1 \\
    $a$ & $a$ & $a$ & $b$ & 1 \\
    $b$ & $b$ & $b$ & $b$ & 1 \\
    1 & 1 & 1 & 1 & 1
\end{tabular} \hspace{.5cm}
\begin{tabular}{r|rrrr}
$\cdot$ & 0 & $a$ & $b$ & 1\\
\hline
    0 & 0 & 0 & 0 & 0 \\
    $a$ & 0 & 0 & 0 & $a$ \\
    $b$ & 0 & 0 & 0 & $b$ \\
    1 & 0 & $a$ & $b$ & 1
\end{tabular}
\caption{A commutative semiring with elements~$a$, $b$ such that the transposition
$f=(a\quad b)$ is a brachy-au\-to\-mor\-phism satisfying $f(a+b)\neq f(a)+f(b)$}\label{Tab:BrachySemiring}
\end{table}

\subsection{Near-rings}\label{Su:Near-rings}
A \emph{right near-ring} is defined the same way as a ring, except that the addition is not assumed to be commutative (it still remains a group) and the multiplication is not assumed to be left distributive on the addition (we keep the right distributivity $(x+y)z=xz+yz$).

Table~\ref{Tab:NonAssocNearRing} represents a (unital) right near-ring with commutative addition, with a brachy-automorphism $f=\begin{pmatrix}6 & 8\end{pmatrix}\begin{pmatrix}7 & 9\end{pmatrix}$ (i.e., expressed as a product of transpositions) and elements~$a=2$, $b=4$ such that 
$f(a)=a$, $f(b)=b$, but $f(a+b)\neq a+b$; hence in this example, the addition is not determined by the multiplication and the successor function.
This near-ring satisfies the identity $2x=0$, so its additive group is $(\ZZ/2\ZZ)^4$.
It also satisfies the identity~$x0=0$ (which does not hold in all right near-rings).

\begin{table}[htb]  \centering % size 16
\begin{tabular}{r|rrrrrrrrrrrrrrrr}
$+$ & 0 & 1 & 2 & 3 & 4 & 5 & 6 & 7 & 8 & 9 & 10 & 11 & 12 & 13 & 14 & 15\\
\hline
    0 & 0 & 1 & 2 & 3 & 4 & 5 & 6 & 7 & 8 & 9 & 10 & 11 & 12 & 13 & 14 & 15 \\
    1 & 1 & 0 & 3 & 2 & 5 & 4 & 7 & 6 & 9 & 8 & 11 & 10 & 13 & 12 & 15 & 14 \\
    2 & 2 & 3 & 0 & 1 & 6 & 7 & 4 & 5 & 10 & 11 & 8 & 9 & 14 & 15 & 12 & 13 \\
    3 & 3 & 2 & 1 & 0 & 7 & 6 & 5 & 4 & 11 & 10 & 9 & 8 & 15 & 14 & 13 & 12 \\
    4 & 4 & 5 & 6 & 7 & 0 & 1 & 2 & 3 & 12 & 13 & 14 & 15 & 8 & 9 & 10 & 11 \\
    5 & 5 & 4 & 7 & 6 & 1 & 0 & 3 & 2 & 13 & 12 & 15 & 14 & 9 & 8 & 11 & 10 \\
    6 & 6 & 7 & 4 & 5 & 2 & 3 & 0 & 1 & 14 & 15 & 12 & 13 & 10 & 11 & 8 & 9 \\
    7 & 7 & 6 & 5 & 4 & 3 & 2 & 1 & 0 & 15 & 14 & 13 & 12 & 11 & 10 & 9 & 8 \\
    8 & 8 & 9 & 10 & 11 & 12 & 13 & 14 & 15 & 0 & 1 & 2 & 3 & 4 & 5 & 6 & 7 \\
    9 & 9 & 8 & 11 & 10 & 13 & 12 & 15 & 14 & 1 & 0 & 3 & 2 & 5 & 4 & 7 & 6 \\
    10 & 10 & 11 & 8 & 9 & 14 & 15 & 12 & 13 & 2 & 3 & 0 & 1 & 6 & 7 & 4 & 5 \\
    11 & 11 & 10 & 9 & 8 & 15 & 14 & 13 & 12 & 3 & 2 & 1 & 0 & 7 & 6 & 5 & 4 \\
    12 & 12 & 13 & 14 & 15 & 8 & 9 & 10 & 11 & 4 & 5 & 6 & 7 & 0 & 1 & 2 & 3 \\
    13 & 13 & 12 & 15 & 14 & 9 & 8 & 11 & 10 & 5 & 4 & 7 & 6 & 1 & 0 & 3 & 2 \\
    14 & 14 & 15 & 12 & 13 & 10 & 11 & 8 & 9 & 6 & 7 & 4 & 5 & 2 & 3 & 0 & 1 \\
    15 & 15 & 14 & 13 & 12 & 11 & 10 & 9 & 8 & 7 & 6 & 5 & 4 & 3 & 2 & 1 & 0
\end{tabular} \hspace{.5cm}
\begin{tabular}{r|rrrrrrrrrrrrrrrr}
$\cdot$ & 0 & 1 & 2 & 3 & 4 & 5 & 6 & 7 & 8 & 9 & 10 & 11 & 12 & 13 & 14 & 15\\
\hline
    0 & 0 & 0 & 0 & 0 & 0 & 0 & 0 & 0 & 0 & 0 & 0 & 0 & 0 & 0 & 0 & 0 \\
    1 & 0 & 1 & 2 & 3 & 4 & 5 & 6 & 7 & 8 & 9 & 10 & 11 & 12 & 13 & 14 & 15 \\
    2 & 0 & 2 & 0 & 0 & 0 & 0 & 0 & 0 & 0 & 0 & 0 & 0 & 0 & 0 & 0 & 12 \\
    3 & 0 & 3 & 2 & 3 & 4 & 5 & 6 & 7 & 8 & 9 & 10 & 11 & 12 & 13 & 14 & 3 \\
    4 & 0 & 4 & 0 & 0 & 0 & 0 & 0 & 0 & 0 & 0 & 0 & 0 & 0 & 0 & 0 & 4 \\
    5 & 0 & 5 & 2 & 3 & 4 & 5 & 6 & 7 & 8 & 9 & 10 & 11 & 12 & 13 & 14 & 11 \\
    6 & 0 & 6 & 0 & 0 & 0 & 0 & 0 & 0 & 0 & 0 & 0 & 0 & 0 & 0 & 0 & 8 \\
    7 & 0 & 7 & 2 & 3 & 4 & 5 & 6 & 7 & 8 & 9 & 10 & 11 & 12 & 13 & 14 & 7 \\
    8 & 0 & 8 & 0 & 0 & 0 & 0 & 0 & 0 & 0 & 0 & 0 & 0 & 0 & 0 & 0 & 6 \\
    9 & 0 & 9 & 2 & 3 & 4 & 5 & 6 & 7 & 8 & 9 & 10 & 11 & 12 & 13 & 14 & 9 \\
    10 & 0 & 10 & 0 & 0 & 0 & 0 & 0 & 0 & 0 & 0 & 0 & 0 & 0 & 0 & 0 & 10 \\
    11 & 0 & 11 & 2 & 3 & 4 & 5 & 6 & 7 & 8 & 9 & 10 & 11 & 12 & 13 & 14 & 5 \\
    12 & 0 & 12 & 0 & 0 & 0 & 0 & 0 & 0 & 0 & 0 & 0 & 0 & 0 & 0 & 0 & 2 \\
    13 & 0 & 13 & 2 & 3 & 4 & 5 & 6 & 7 & 8 & 9 & 10 & 11 & 12 & 13 & 14 & 13 \\
    14 & 0 & 14 & 0 & 0 & 0 & 0 & 0 & 0 & 0 & 0 & 0 & 0 & 0 & 0 & 0 & 14 \\
    15 & 0 & 15 & 2 & 3 & 4 & 5 & 6 & 7 & 8 & 9 & 10 & 11 & 12 & 13 & 14 & 1
\end{tabular}
\caption{A right near-ring with commutative addition, with elements $a=2$, $b=4$ and a brachy-automorphism $f=(6\quad 8)(7\quad 9)$ satisfying $f(a)=a$, $f(b)=b$, and $f(a+b)\neq a+b$}
\label{Tab:NonAssocNearRing}
\end{table}

\subsection{Cancellative semirings}\label{Su:CancSemiRings}
In constrast to the example given in Sub\-sec\-tion~\ref{Su:Semiring}, the context of (additively) cancellative semirings (i.e., $x+z=y+z\Rightarrow x=y$) allows many addability results, originally formulated for rings, to extend to cancellative semirings (though not trivially: while every cancellative semiring~$R$ canonically embeds into its universal ring~$\ol{R}$, it is not clear whether a brachymorphism on~$R$ would extend to a brachymorphism on~$\ol{R}$).
A sample of such extensions is given as follows:

\begin{enumerater}
\item\textbf{Lemma~\ref{L:CharAdd}}
(first-order characterization of summable tuples): change ``ring'' to ``semiring'' (or ``cancellative semiring'', or anything reasonable), same proof.

\item\textbf{Proposition~\ref{P:ClosAdd}} (summability is preserved under homomorphic image and finite direct product): change ``ring'' to ``cancellative semiring'', same proof.

\item\textbf{Lemma~\ref{L:pq2Add}} (summability formula for $xyx=x^2y$): change ``ring'' to ``cancellative semiring'', modify the proof \emph{mutatis mutandis} (e.g., change $(x+y-z)z=0$ to $(x+y)z=z^2$).
Corollary~\ref{C:pq2Add} remains valid for cancellative semirings.

\item\textbf{Theorem~\ref{T:IntClos}}
\label{IntClos}
(integral closedness of~$\Add{R}$ in~$R$): change ``ring'' to ``cancellative semiring'', extend the definition of ``right integral over~$A$'' by enabling
 \[
 x^n+\sum\nolimits_{k<n}b_kx^k=\sum\nolimits_{k<n}a_kx^k
 \]
with all $a_k,b_k\in A$.
Change the equation $zx^n=\sum_{k<n}a_kx^k+(a_n+y)x^n$ to $\sum_{k<n}b_kx^k+zx^n=\sum_{k<n}a_kx^k+(a_n+y)x^n$.
The rest of the proof goes along the same lines as in the ring case.

\item\textbf{Section~\ref{S:ClAddR}}.
Owing to~\eqref{IntClos} above, all other results of Section~\ref{S:ClAddR} can be extended the same way, with the exceptions of Corollaries~\ref{C:abs2z=x+y} and~\ref{C:pireg}, for it is not clear whether regular implies addable in that more general cancellative semiring context.
For the desired extension of Corollary~\ref{C:Komatsu}, note that every brachynomial $t(x,y)$ has a unique monomial of largest degree $d$, and that $d>0$ in all nontrivial cases.
Evaluating $t(n\cdot1,n\cdot1)$ for a large enough integer~$n$, it follows that $m\cdot1=0$ for some positive integer~$m$; thus our semiring is in fact a ring.

\item\textbf{Theorem~\ref{T:z+xzy}} ($f(1+x)=1+f(x)$ and $f(z+xzy)=f(z)+f(x)f(z)f(y)$ together entail additivity): change ``ring'' to ``cancellative semiring'', modify the proof \emph{mutatis mutandis}.

\end{enumerater}

\subsection{More open problems}\label{Su:OpenPbs}

We know from Corollary~\ref{C:pireg} (see also Remark~\ref{Rk:pireg}) that every brachymorphism from a finite ring is additive.
For (finite or infinite) non-associative rings we do not know the answer.

\begin{problem}
Is every brachymorphism from a (unital) non-associative ring additive?
\end{problem}

We stated at the beginning of the paper the basic unsolved problem, whether every brachymorphism between rings is additive.
By Proposition~\ref{P:ClosAdd}\eqref{Quot2Summ}, this is equivalent to saying that the pair $(x,y)$ is summable in the free associative ring~$\ZZ\seq{x,y}$. 
A strong negative solution to that problem would follow from a positive solution to the following problem.

\begin{problem}\label{Pb:ExoticEval}
Let~$\kk$ be a field.
Does there exist a brachymorphism $f\colon\kk\seq{x,y}\to\kk$ such that $f(x)=f(y)=0$ whereas $f(x+y)=1$?
\end{problem}

Problem~\ref{Pb:ExoticEval} seems to have something to do with the one, stated as Conjecture~3, page~122 in Cohn~\cite{Cohn1974}, whether no nonconstant element of~$\kk\seq{\gS}$ generates~$\kk\seq{\gS}$ as a two-sided ideal.
(This problem is also mentioned in Bergman~\cite{Berg2021}, which is focused on an analogue with group algebras over free groups.)
For example, if $u_1av_1+u_2av_2=1$ for some $u_i,v_i\in\kk\seq{x,y}$ (which is a strong way of ensuring that~$a$ generates~$\kk\seq{x,y}$ as a two-sided ideal), then $f(a)\neq0$ for any brachymorphism~$f$ solving Problem~\ref{Pb:ExoticEval}.
On the other hand, the argument presented in Makar-Limanov \cite[\S~5]{Makar1985} shows that if the characteristic of~$\kk$ is zero, then $u_1av_1+u_2av_2=1$ cannot occur unless~$a$ is constant.

\section*{Acknowledgments}\label{S:Acknow}
The author thanks Ken Goodearl, Hannes Thiel, and, most of all, George Berg\-man, for many instructive comments on this preprint, which led to a number of improvements.

%\bibliographystyle{amsplain}
%\bibliography{TM_biblio}

\providecommand{\noopsort}[1]{}\def\cprime{$'$}
  \def\polhk#1{\setbox0=\hbox{#1}{\ooalign{\hidewidth
  \lower1.5ex\hbox{`}\hidewidth\crcr\unhbox0}}}
  \providecommand{\bysame}{\leavevmode\hbox to3em{\hrulefill}\thinspace}
\providecommand{\MR}{\relax\ifhmode\unskip\space\fi MR }
\providecommand{\MRhref}[2]{%
  \href{http://www.ams.org/mathscinet-getitem?mr=#1}{#2}
}
\providecommand{\href}[2]{#2}

\end{document}